\documentclass{birkjour}
\newtheorem{theorem}{Theorem}[section]
\newtheorem{corollary}[theorem]{Corollary}
\newtheorem{lemma}[theorem]{Lemma}
\newtheorem{proposition}[theorem]{Proposition}

\theoremstyle{remark}

\theoremstyle{remark}

\theoremstyle{remark}
\newtheorem{remark}[theorem]{Remark}

\newcommand{\R}{{\mathbb R}}
\newcommand{\N}{{\mathbb N}}

\newcommand{\di}{\partial}
\newcommand{\la}{\langle}
\newcommand{\ra}{\rangle}
\newcommand{\wick}[1]{{:}\omega^{\otimes #1}{:}}

\def\theequation{\thesection.\arabic{equation}}
\numberwithin{equation}{section}

\begin{document}

%
%
%
%
%
%
%
%
%

\title[Wick calculus for noncommutative $q$-white noise]
 {Wick calculus for noncommutative white noise corresponding to $q$-deformed commutation relations}

\author[U. C. Ji]{Un Cig Ji}
\address{Research Institute of Mathematical Finance\\
Chungbuk National University\\
Cheongju 28644\\
Republic of Korea}
\email{uncigji@chungbuk.ac.kr}

\thanks{The first author (UCJ) was supported by Basic Science Research Program
through the NRF funded by the MEST (No. NRF-2016R1D1A1B01008782).}

\author[E. Lytvynov]{Eugene Lytvynov}
\address{Department of Mathematics\\
Swansea University\\
Singleton Park\\
Swansea SA2 8PP\\
U.K.}
\email{e.lytvynov@swansea.ac.uk}



\date{July 10, 2016}
\dedicatory{Dedicated to Professor Marek Bo\.zejko on the occasion of his 70th birthday}


\begin{abstract}
We derive the Wick calculus for test and generalized functionals
of noncommutative white noise corresponding to $q$-deformed commutation relations with $q\in(-1,1)$.
We construct a Gel'fand triple centered at the $q$-deformed Fock space
in which both  the test, nuclear space and its dual space are algebras
with respect to the addition and the Wick multiplication.
Furthermore, we prove a V\r{a}ge-type inequality for the Wick product on the dual space.
\end{abstract}

\maketitle

\section{Itroduction}
White noise analysis is a theory of test and generalized functionals
of the time derivative of Brownian motion---white noise.
It has found numerous applications, in particular in  mathematical physics  and  mathematical finance,
see e.g.\ \cite{BK,DOP,HKPS,HOUZ,Obata} and the references therein.
This theory is based on the Wiener--It\^o--Segal isomorphism
between the $L^2$-space with respect to a Gaussian measure and the symmetric Fock space.
Consider a linear functional $\langle \omega,\varphi\rangle$,
the dual pairing of white noise $\omega$ and a test function $\varphi$,
and then $\langle \omega,\varphi\rangle$ is a Gaussian random variable with mean $0$ and variance $|\varphi|_{L^2}^2$.
Under the Wiener--It\^o--Segal isomorphism
the operator of multiplication by $\langle \omega,\varphi\rangle$ in the $L^2$-space
becomes the operator $a^+(\varphi)+a^-(\varphi)$ in the symmetric Fock space,
which is called the quantum decomposition of the Gaussian random variable $\langle \omega,\varphi\rangle$.
Here $a^+(\varphi)$ and $a^-(\varphi)$ are creation and annihilation operators, respectively,
and satisfy the canonical commutation relations.

A crucial  technical tool in  white noise analysis is the Wick product, denoted by~$\diamond$.
Unlike the point-wise product, the Wick product
is a well-defined operation on certain spaces of generalized functionals of $\omega$,
like for example, the Hida space \cite{HKPS,Obata} or the Kondratiev space \cite{K,KLS,HOUZ}.
In the latter space, one can additionally treat an infinite Wick-power series:
\begin{equation}\label{hijg8turt}
\phi^\diamond(F)=\sum_{n=0}^\infty a_n F^{\diamond n},
\end{equation}
where $F$ is an element of the Kondratiev space
and $\phi(z)=\sum_{n=0}^\infty a_n z^n$ is an analytic function of a complex variable $z$ from a neighborhood of zero.
In particular, one shows that a distribution $F$ has an inverse with respect to  Wick product
if and only if the expectation of $F$ (in a generalized sense) is not equal to zero.
Generalized functionals $\phi^\diamond(F)$ are used in stochastic differential equations involving the Wick product, e.g.\ \cite{HOUZ,DOP}.
While the original proof of the convergence of the series given as in \eqref{hijg8turt}
was done through the analytic characterization theorem in terms of the so-called $S$-transform \cite{K,KLS},
an alternative proof of this fact is based on the V\r{a}ge inequality \cite{HOUZ,V}.

Let us also recall that, in terms of the symmetric Fock space,
the Wick product corresponds to the  symmetric tensor product.
Equivalently, in terms of the creation and annihilation operators acting on the symmetric Fock space,
the Wick product $\diamond$ means normal ordering,
i.e., moving the creation operators to the left and the annihilation operators to the right, see e.g. \cite{HKPS}.
For example, for test functions $\varphi$ and $\psi$,
\begin{align}
&(a^+(\varphi)+a^-(\varphi))\diamond(a^+(\psi)+a^-(\psi))\notag\\
&\quad=a^+(\varphi)a^+(\psi)+a^+(\varphi)a^-(\psi)+a^+(\psi)a^-(\varphi)+
a^-(\varphi)a^-(\psi).
\label{ufrt7ur}
\end{align}

In the framework of the free probability,
Alpay and Salomon \cite{AS} (see also \cite{AJS})
constructed a noncommutative analog of the Kondratiev space
and proved a V\r{a}ge-type inequality on it.
More precisely, they constructed a Gel'fand triple with the center space being the full Fock space,
defined the Wick product on it as the usual tensor product,
and proved that the tensor product on the dual space satisfies a V\r{a}ge-type inequality.

For $q\in (-1,1)$, Bo\.zejko and Speicher \cite{BS}
constructed a representation of the $q$-deformed commutation relation:
\begin{equation}\label{buygy8ugt}
 a^-(\varphi)a^+(\psi)=qa^+(\varphi)a^-(\psi)+( \varphi,\psi)_\mathcal{H}.
\end{equation}
Here $\varphi$ and $\psi$ are elements of a Hilbert space $\mathcal{H}$.
Note the absence of commutation relations between $a^+(\varphi)$ and $a^+(\psi)$,
respectively between $a^-(\varphi)$ and $a^-(\psi)$.
The operators $a^+(\varphi)$ and $a^-(\varphi)$ are realized
as the creation and annihilation operators in the $q$-deformed Fock space.
By analogy with the classical case ($q=1$),
the operators $a^+(\varphi)+a^-(\varphi)$
are called a $q$-deformed (noncommutative) Gaussian random variable \cite{BS,BKS}.

Choose $\mathcal{H}=L^2(\mathbb R_+,dx)$, and let $\chi_{[0,t]}$ denote the indicator function of the interval $[0,t]$.
Then the operators $$ B(t)=a^+(\chi_{[0,t]})+a^-(\chi_{[0,t]}),\quad t\ge0,$$
form a $q$-Brownian motion. In the special case $q=0$, we  get a free Brownian motion. In particular, this noncommutative stochastic process has freely independent increments, see e.g.\ \cite{NS}. In the case of a general $q\in(-1,1)$, the $q$-Brownian motion has $q$-independent increments according to the definition of $q$-independence in \cite{A2}, see also \cite{A1}.

Let us informally define a  $q$-white noise by
\[
\omega(t)
:=\frac{d}{dt}\Big|_{t=0}B(t)
= a^+(\delta_t)+a^-(\delta_t),
\]
where $\delta_t$ is the delta function at $t$.
Thus $\omega(t)$ is an operator-valued distribution which satisfies
\[
\langle\omega,\varphi\rangle
=\int_{\mathbb{R}_+}\omega(t)\varphi(t)dt
=a^+(\varphi)+a^-(\varphi)
\]
for any test function $\varphi$.

As shown in \cite{BS},
the normal ordering for the creation and annihilation operators on the $q$-deformed Fock space
means moving the creation operators to the left and the annihilation operators to the right
with the help of the rule $a^-(\varphi)\diamond a^+(\psi)=qa^+(\psi)a^-(\varphi)$.
For example,
\begin{align}
&(a^+(\varphi)+a^-(\varphi))\diamond(a^+(\psi)+a^-(\psi))\notag\\
&\quad=a^+(\varphi)a^+(\psi)+a^+(\varphi)a^-(\psi)+qa^+(\psi)a^-(\varphi)+
a^-(\varphi)a^-(\psi),
\label{udFt7u}
\end{align}
which, of course, becomes \eqref{ufrt7ur} as $q\to1$.
The corresponding Wick product $\diamond$ on the $q$-deformed Fock space
has again the form of the tensor product.
However, for $q\ne0$,
one has to deal with the $q$-deformed scalar product
on tensor powers of a Hilbert space.

The aim of the present paper
is to introduce a proper mathematical framework of $q$-white noise
based on which the Wick calculus for $q$-white noise is available.
For our purpose, we will construct a Gel'fand triple centered at the $q$-deformed Fock space
in which both the test, nuclear space and its dual space are algebras
with respect to the addition and the Wick multiplication,
and furthermore a V\r{a}ge-type inequality holds on the dual space.
In the limit $q\to1$,
we will recover the Kondratiev space of generalized functionals of white noise,
while for $q=0$, in a special case, we will recover the space from \cite{AS}.
We note that the V\r{a}ge-type inequalities derived in this paper
for both the classical case ($q=1$, see \cite{V}) and the free case ($q=0$, see \cite{AS})
differ from the known ones, and actually have a simpler form.

The paper is organized as follows.
In Section~\ref{tydr6},
we discuss orthogonalization of polynomials of $q$-white noise $\omega$
in the corresponding noncommutative $L^2$-space with respect to the vacuum expectation.
In Section~\ref{jigt8i7},
we introduce the Wick product and discuss how it is related to the orthogonal polynomials.
The main results of the paper are in Sections~\ref{vcytd657i} and~\ref{jkgt7u6r7}.
In Section~\ref{vcytd657i},
we construct spaces of noncommutative test functionals of $\omega$
which are algebras under the addition and the Wick multiplication.
Finally, in Section~\ref{jkgt7u6r7},
we prove that the corresponding dual spaces are also algebras under the addition and the Wick multiplication,
and furthermore a V\r{a}ge-type inequality holds on them.
Just like in the classical and free cases,
we get immediate consequences of this result,
like well-definiteness of an infinite Wick-power series given as in \eqref{hijg8turt}.

\section{Orthogonal  polynomials of $q$-white noise}\label{tydr6}

Let us first recall a representation of the $q$-commutation relations, see \cite{BS} for details.
Let $\mathcal{H}$ be a real separable Hilbert space with the scalar product $(\cdot,\cdot)_\mathcal{H}$.
Let $q\in(-1,1)$.
For each $n\in\N$, we define a bounded linear operator
$P_q^{(n)}:\mathcal{H}^{\otimes n}\to\mathcal{H}^{\otimes n}$ by
\begin{equation}\label{hvftdf7yihyu}
 P_q^{(n)} f_1\otimes\dots\otimes f_n
 :=\sum_{\pi\in S_n}q^{\operatorname{inv}(\pi)}f_{\pi(1)}\otimes\dots\otimes f_{\pi(n)}
\end{equation}
for $f_1,\dots,f_n\in \mathcal{H}$,
where $S_n$ is the symmetric group of order $n$
and, for each permutation $\pi\in S_n$,
$\operatorname{inv}(\pi)$ denotes the number of inversions in $\pi$,
i.e., the number of all pairs $(i,j)$, $1\le i<j\le n$, such that $\pi(i)>\pi(j)$.
In this paper, for convenience, we use that $0^0:=1$ when $q=0$.
Then the operator $P_q^{(n)}$ is strictly positive (see \cite{BS}).
We define $\mathcal{F}^{(n)}_q(\mathcal{H})$
as the real separable Hilbert space which coincides with $\mathcal{H}^{\otimes n}$ as a set
and has scalar product
\[
( f^{(n)},g^{(n)})_{\mathcal{F}^{(n)}_q(\mathcal{H})}
:=( P_q^{(n)} f^{(n)},g^{(n)})_{\mathcal{H}^{\otimes n}}\,.
\]
We also set $\mathcal{F}^{(0)}_q(\mathcal{H}):=\R$.
We define the $q$-Fock space
\[
\mathcal{F}_q(\mathcal{H})
:=\bigoplus_{n=0}^\infty \mathcal{F}^{(n)}_q(\mathcal{H}).
\]
As usual, we will identify any $f^{(n)}\in \mathcal{F}^{(n)}_q(\mathcal{H})$
with the corresponding element of $\mathcal{F}_q(\mathcal{H})$.

For each $\varphi\in\mathcal{H}$,
we consider bounded linear operators $a^+(\varphi)$ and $a^-(\varphi)$ in $\mathcal{F}_q(\mathcal{H})$
which are defined by
\begin{align*}
a^+(\varphi)f^{(n)}&:=\varphi\otimes f^{(n)},\\
a^-(\varphi) f_1\otimes\dots\otimes f_n
   &:=\sum_{i=1}^n q^{i-1}(\varphi,f_i)_\mathcal{H}\, f_1\otimes\dots\otimes \check f_{i}\otimes \dots\otimes f_n
\end{align*}
for $f^{(n)}\in \mathcal{F}^{(n)}_q(\mathcal{H})$ and $f_1,\dots,f_n\in \mathcal{H}$,
where $\check f_i$ denotes the absence of $f_i$.
Thus, $a^+(\varphi)$ and $a^-(\varphi)$
are called the creation and annihilation operators, respectively,
and are adjoint of each other in $\mathcal{F}_q(\mathcal{H})$.

Although our results will hold for any infinite dimensional separable Hilbert space $\mathcal{H}$,
it will be convenient to think of $\mathcal{H}$ as an $L^2$-space $L^2(X,\sigma)$.
Here $X$ is a separable topological space with a $\sigma$-finite non-atomic Borel measure $\sigma$.
For $x\in X$, we introduce, at least informally, the annihilation operator $\di_x$
and the creation operator $\di^\dag_x$ at point $x$ so that,
for each $\varphi\in\mathcal{H}$,
\begin{align}
a^-(\varphi)&=\int_X \di_x\,\varphi(x)\,d\sigma(x),\label{gftyrt5er67}\\
a^+(\varphi)&=\int_X \di^\dag_x\,\varphi(x)\,d\sigma(x)\label{fYUGUYW}.
\end{align}
Thus, for $f^{(n)}\in\mathcal{F}_q^{(n)}(\mathcal{H})$,
it holds that
\begin{align*}
(\di_x f^{(n)})(x_1,\dots,x_{n-1})&=\sum_{i=1}^n q^{i-1}f^{(n)}(x_1,\dots,x_{i-1},x,x_i,\dots,x_{n-1}),\\
\di_x^\dag f^{(n)}&=\delta_x\otimes f^{(n)},
\end{align*}
where $\delta_x$ is the delta function at $x$.
A rigorous meaning of the above formulas is given
through the corresponding quadratic forms, cf.\ e.g.\ \cite{RS2}.

We define the $q$-white noise by
\[
\omega(x)=\di^\dag_x +\di_x,\quad x\in X.
\]
Thus, for each $\varphi\in \mathcal{H}$,
we obtain a bounded self-adjoint operator in $\mathcal{F}_q(\mathcal{H})$ by setting
\[
\la \omega,\varphi\ra:=\int_X\omega(x)\varphi(x)\,d\sigma(x)
=a^+(\varphi)+a^-(\varphi).
\]

We denote by $\mathcal{H}^{\otimes_{\mathbf a}n}$
the $n$-th algebraic tensor power of $\mathcal{H}$.
That is, $\mathcal{H}^{\otimes_{\mathbf a}n}$ is the subset of $\mathcal{H}^{\otimes n}$
which is the linear span of the vectors $f_1\otimes\dots\otimes f_n$ with $f_1,\dots,f_n\in\mathcal{H}$.
Let $\mathcal{P}$ denote the real algebra generated by the operators $\la \omega,\varphi\ra$ with $\varphi\in\mathcal{H}$.
We denote
\[
\la \omega^{\otimes n},f_1\otimes\dots\otimes f_n\ra
:= \la \omega,f_1\ra\dotsm \la \omega,f_n\ra.
\]
Extending by linearity,
we define a (noncommutative) monomial $\la \omega^{\otimes n},f^{(n)}\ra$
for each $f^{(n)}\in \mathcal{H}^{\otimes_{\mathbf a}n}$.
Evidently, $\mathcal{P}$ consists of all (noncommutative) polynomials in $\omega$
which are of the form:
\begin{equation}\label{hfur8rt}
P= f^{(0)}+\sum_{i=1}^n \la \omega^{\otimes i},f^{(i)}\ra,
\quad f^{(0)}\in\R,\ f^{(i)}\in \mathcal{H}^{\otimes_{\mathbf a}i}.
\end{equation}

We define the vacuum expectation on  $\mathcal{P}$ by
\[
\mu (P):=( P\Omega,\Omega)_{\mathcal{F}_q(\mathcal{H})},\quad P\in\mathcal{P},
\]
where $\Omega:=(1,0,0,\dots)$ is the vacuum vector in $\mathcal{F}_q(\mathcal{H})$.
We define an inner product
\begin{equation}\label{hgftyr7e7r}
(P_1,P_2)_{L^2(\mu)}
:=\mu(P_2^*P_1)
=( P_1\Omega,P_2\Omega)_{\mathcal{F}_q(\mathcal{H})},
\quad P_1,P_2\in\mathcal{P}.
\end{equation}
Note that, for each $P\in\mathcal{P}$ with $P\ne0$,
we have $(P,P)_{L^2(\mu)}>0$.
Hence, we can define a noncommutative $L^2$-space $L^2(\mu)$
to be the real Hilbert space obtained as the closure of $\mathcal{P}$
with respect to the norm induced by the inner product $(\cdot,\cdot)_{L^2(\mu)}$.
In particular, $\mathcal{P}$ is a dense subset of $L^2(\mu)$.

The following proposition immediately follows from the definition of $L^2(\mu)$.

\begin{proposition}\label{gdyde6y}
For each $P\in\mathcal{P}$, define $IP:=P\Omega$.
Then, $I$ is extended by the continuity to a unitary operator
$I:L^2(\mu)\to\mathcal{F}_q(\mathcal{H})$.
Furthermore, under the action of $I$,
the operator of the left multiplication by $\la\omega,\varphi\ra$ in $L^2(\mu)$
(denoted by $L_{\la\omega,\varphi\ra}$) becomes $\la\omega,\varphi\ra$, i.e.,
\begin{equation}\label{jklhiugti8t}
I L_{\la\omega,\varphi\ra} I^{-1}=\la\omega,\varphi\ra.
\end{equation}
\end{proposition}

For $k\in\N_0=\mathbb{N}\cup\{0\}$,
we denote by $\mathcal{P}_k$ the subset of $\mathcal{P}$
consisting of all polynomials of order $\le k$,
i.e., all $P\in\mathcal{P}$ given as in \eqref{hfur8rt} with $n\le k$.
Let $\mathcal{MP}_k$ denote the closure of $\mathcal{P}_k$ in $L^2(\mu)$
(measurable polynomials of order $k$).
Let
\[
\mathcal {OP}_k=\mathcal {MP}_k\ominus\mathcal {MP}_{k-1}
\]
(orthogonal polynomials of order $k$).
Here $\ominus$ denotes orthogonal difference in $L^2(\mu)$.
Since $\mathcal{P}$ is dense in $L^2(\mu)$,
we get the orthogonal decomposition:
\[
L^2(\mu)=\bigoplus_{n=0}^\infty \mathcal {OP}_n.
\]

For each $f^{(n)}\in \mathcal{H}^{\otimes_{\mathbf a}n}$, we denote by
$\la \wick{n},f^{(n)}\ra$ the orthogonal projection of the monomial $\la \omega^{\otimes n},f^{(n)}\ra$ onto $\mathcal {OP}_n$.

\begin{proposition}\label{hjuyfuyf}
For each $f^{(n)}\in \mathcal{H}^{\otimes_{\mathbf a}n}$,
we have
\[
I \la\wick{n}, f^{(n)}\ra=f^{(n)}.
\]
\end{proposition}

\begin{proof}
It can be easily checked that, for each $n$,
\[
I\mathcal{P}_n=\bigoplus_{k=0}^n \mathcal{H}^{\otimes_{\mathbf a}k}.
\]
Therefore,
\[
I\,\mathcal {MP}_n=\bigoplus_{k=0}^n \mathcal{F}^{(k)}_q(\mathcal{H}),
\]
and so
\[
 I\,\mathcal {OP}_n=\mathcal{F}^{(n)}_q(\mathcal{H}).
\]
Hence, $I \la \wick{n},f^{(n)}\ra$ is the orthogonal projection in $\mathcal{F}_q(\mathcal{H})$ of
\[
I \la \omega^{\otimes n},f^{(n)}\ra=\la \omega^{\otimes n},f^{(n)}\ra\Omega
\]
onto $\mathcal{F}^{(n)}_q(\mathcal{H})$.
But the latter vector is equal to $f^{(n)}$.
\end{proof}

\begin{corollary}\label{jgyiyu}
The following recursive formula holds:
\begin{align}
\la {:}\omega{:},f\ra
&=\la \omega,f\ra,\quad f\in\mathcal{H},\label{vgfytdfyd}\\
\la \wick{n},f_1\otimes\dots\otimes f_n\ra
&=\la \omega,f_1\ra\la \wick{(n-1)},f_2\otimes\dotsm\otimes f_n\ra\notag\\
&\qquad-\la \wick{(n-2)},a^-(f_1)f_2\otimes\dots\otimes f_n\ra\label{khgiytgit}
\end{align}
for $ n\ge2$ and $f_1,\dots,f_n\in\mathcal{H}$.
In particular, for each $f^{(n)}\in \mathcal{H}^{\otimes_{\mathbf a}n}$,
we have $\la \wick{n},f^{(n)}\ra\in\mathcal{P}$.
\end{corollary}

\begin{proof}
By Proposition~\ref{hjuyfuyf}, for $f\in\mathcal{H}$
\[
I\la {:}\omega{:},f\ra= f=\la \omega,f\ra\Omega=I\la \omega,f\ra,
\]
so \eqref{vgfytdfyd} holds.
Next, for any $f_1,\dots,f_n\in\mathcal{H}$ ($n\ge2$),
we get, by Propositions~\ref{gdyde6y} and \ref{hjuyfuyf},
\begin{align*}
I\big(\la \omega,f_1\ra\la \wick{(n-1)},f_2\otimes\dotsm\otimes f_n\ra\big)
&=\la \omega,f_1\ra   f_2\otimes\dotsm\otimes f_n\\
&=\left(a^+(f_1)+a^-(f_1)\right)f_2\otimes\dotsm\otimes f_n\\
&=I \la \wick{n},f_1\otimes\dots\otimes f_n\ra\\
&\quad+\la \wick{(n-2)},a^-(f_1)f_2\otimes\dots\otimes f_n\ra,
\end{align*}
from which \eqref{khgiytgit} follows.
\end{proof}

\begin{corollary}\label{hgigtitg8i}
The set $\mathcal{P}$ consists of all (noncommutative) polynomials of the form:
\begin{equation}\label{guftyrf7yr}
P= f^{(0)}+\sum_{i=1}^n \la\wick{i}, f^{(i)}\ra,\quad f^{(0)}\in\R,\ f^{(i)}\in \mathcal{H}^{\otimes_{\mathbf a}i}.
\end{equation}
\end{corollary}

\begin{proof}
By Corollary \ref{jgyiyu}, each polynomial of the form \eqref{guftyrf7yr} belongs to $\mathcal{P}$.
So we only need to prove that, for any $n\in\N$, $f_1,\dots,f_n\in\mathcal{H}$,
the monomial $\la \omega^{\otimes n},f_1\otimes\dots\otimes f_n\ra $
can be represented in the form \eqref{guftyrf7yr}.
But this can be easily shown by induction.
Indeed, for $n=1$, this follows from \eqref{vgfytdfyd}.
Assume that the statement is true for $1,2,\dots,n$.
For any $f_1,\dots,f_{n+1}\in\mathcal{H}$, by Corollary~\ref{jgyiyu},
\[
\la \omega^{\otimes(n+1)},f_1\otimes\dots\otimes f_{n+1}\ra
  -\la \wick{(n+1)},f_1\otimes\dots\otimes f_{n+1}\ra
\]
is a polynomial from $\mathcal{P}$ of order $\le n$.
Hence, the statement follows by the induction assumption.
\end{proof}

\section{Wick product}\label{jigt8i7}

We will now introduce the Wick product.
Let $\mathcal{W}$ be the linear span of the bounded linear operators
in $\mathcal{F}_q(\mathcal{H})$ of the form:
\begin{equation}\label{ugfuygf}
a^+(f_1)\dotsm a^+(f_n)a^-(g_1)\dotsm a^-(g_m)
\end{equation}
for $n,m\in \N_0$, $f_1,\dots,f_n,g_1,\dots,g_m\in\mathcal{H}$, and the identity operator.
Note that the operator in \eqref{ugfuygf} is in the normal ordered form,
i.e., all creation operators are to the left of all annihilation operators.
We define the Wick product $\diamond$ on $\mathcal{W}$ by setting
\begin{align}
& \big(a^+(f_1)\dotsm a^+(f_n)a^-(g_1)\dotsm a^-(g_m)\big)\diamond
\big(a^+(\varphi_1)\dotsm a^+(\varphi_k)a^-(\psi_1)\dotsm a^-(\psi_l)\big)\notag\\
&\qquad\qquad :=q^{km}\,a^+(f_1)\dotsm a^+(f_n)a^+(\varphi_1)\dotsm a^+(\varphi_k)\notag\\
&\qquad\qquad\qquad\qquad\qquad \times a^-(g_1)\dotsm a^-(g_m)a^-(\psi_1)\dotsm a^-(\psi_l),
\label{hiygu8ytf}
\end{align}
and by extending this operation by the linearity.
Thus, the Wick product is nothing but
bringing all terms in the usual product,
with the help of the rule $a^-(\varphi)\diamond a^+(\psi)=qa^+(\psi)a^-(\varphi)$
to the normal ordered form.
Clearly, $\mathcal{W}$ is a real algebra with operations of the addition and the Wick multiplication.

It will also be useful to introduce the Wick product for normally ordered products
of creation and annihilation operators at point.
More precisely, we set
\begin{align}
&\big(\di_{x_1}^\dag\dotsm  \di_{x_n}^\dag \di_{s_1}\dotsm\di_{s_m}\big)\diamond\big(\di_{y_1}^\dag\dotsm  \di_{y_k}^\dag \di_{t_1}\dotsm\di_{t_l}\big)\notag\\
&\quad\qquad :=q^{km}\, \di_{x_1}^\dag\dotsm  \di_{x_n}^\dag \di_{y_1}^\dag\dotsm  \di_{y_k}^\dag \di_{s_1}\dotsm\di_{s_m}\di_{t_1}\dotsm\di_{t_l}
\label{hgigti9}
\end{align}
(see \cite{JK} and the references therein),
and extend this definition by linearity.
Thus, the formula given  in \eqref{hiygu8ytf}
is the smeared version of the formula given  in \eqref{hgigti9}.

The following proposition follows from \cite{BKS}
(see, in particular, the last paragraph on p.~137).
Its meaning is that the orthogonalization of polynomials in $L^2(\mu)$
is equivalent to taking the Wick product.

\begin{proposition}\label{hguyfgu}
We have
\[
\wick{n}(x_1,\dots,x_n)= \omega(x_1)\diamond\dots\diamond\omega(x_n),
\]
or in the smeared form
\begin{equation}\label{jigity9i}
\la \wick{n},f^{(n)}\ra
=\int_{X^n} \omega(x_1)\diamond\dots\diamond\omega(x_n)\,f^{(n)}(x_1,\dots,x_n)\,d\sigma(x_1)\dotsm d\sigma(x_n)
\end{equation}
for each $f^{(n)}\in\mathcal{H}^{\otimes_{\mathbf a}n}$.
In particular, for any $f_1,\dots,f_n\in\mathcal{H}$,
\begin{equation}\label{hjguyt8u}
\la \wick{n},f_1\otimes\dots\otimes f_n\ra
=\la \omega,f_1\ra\diamond\dots\diamond \la \omega,f_n\ra.
\end{equation}
\end{proposition}

We denote by $\mathcal{G}_{\mathrm{fin}}^{\mathrm a}(\mathcal{H})$
the linear subspace of $\mathcal{F}_q(\mathcal{H})$
which consists of all finite vectors $(f^{(0)},f^{(1)},\dots,f^{(n)},0,0,\dots)$,
where $f^{(0)}\in\R$, $f^{(i)}\in\mathcal{H}^{\otimes_{\mathrm a}i}$ for $i=1,\dots,n$, $n\in\N$.
As easily seen, $\mathcal{G}_{\mathrm{fin}}^{\mathrm a}(\mathcal{H})$
is an algebra for the addition and the tensor multiplication $\otimes$
with neutral element $\Omega$.

\begin{corollary}\label{oii-0pu}
We have $\mathcal{P}\subset\mathcal{W}$
and for any $P_1,P_2\in\mathcal{P}$,
we get $P_1\diamond P_2\in\mathcal{P}$.
Thus, $\mathcal{P}$ is an algebra for the addition and the Wick multiplication $\diamond$ with neutral element $1$.
Furthermore,
\begin{equation}\label{gytr78or}
I\mathcal{P}
=\mathcal{G}_{\mathrm{fin}}^{\mathrm a}(\mathcal{H})
\end{equation}
and for any $P_1,P_2\in\mathcal{P}$,
\begin{equation}\label{tyed65}
I(P_1\diamond P_2)=IP_1\otimes IP_2.
\end{equation}
In particular, for any $f^{(n)}\in \mathcal{H}^{\otimes_{\mathbf a}n}$
and $g^{(m)} \in \mathcal{H}^{\otimes_{\mathbf a}m}$,
\begin{equation}\label{uit87t8}
\la \wick{n},f^{(n)}\ra \diamond \la \wick m,g^{(m)}\ra=\la \wick{(n+m)},f^{(n)}\otimes g^{(m)}\ra.
\end{equation}
\end{corollary}

\begin{proof}
Formula \eqref{gytr78or} follows from Proposition~\ref{hjuyfuyf} and Corollary~\ref{hgigtitg8i}.
The inclusion $\mathcal{P}\subset\mathcal{W}$
follows Corollary~\ref{hgigtitg8i} and  Proposition~\ref{hguyfgu},
see in particular formula \eqref{hjguyt8u}.
Formula \eqref{uit87t8} follows from \eqref{jigity9i} (equivalently from \eqref{hjguyt8u}).
Formula \eqref{uit87t8} and Corollary~\ref{hgigtitg8i} imply that,
for any $P_1,P_2\in\mathcal{P}$, we have $P_1\diamond P_2\in\mathcal{P}$.
By Proposition~\ref{hjuyfuyf} and \eqref{uit87t8},
\[
I\large(\la \wick{n},f^{(n)}\ra \diamond \la \wick m,g^{(m)}\ra\large)=f^{(n)}\otimes g^{(m)},
\]
from which formula \eqref{tyed65} follows.
\end{proof}

\section{Algebras for Wick multiplication}\label{vcytd657i}

Our next aim is to extend the Wick multiplication from $\mathcal{P}$
to a wider class of elements of $L^2(\mu)$.
By Corollary~\ref{oii-0pu},
this problem is equivalent to extending the usual tensor product
from $\mathcal{G}_{\mathrm{fin}}^{\mathrm a}(\mathcal{H})$
to a wider class of vectors from the $q$-Fock space $\mathcal{F}_q(\mathcal{H})$.
Note that $\mathcal{F}_q(\mathcal{H})$ is not closed under
the tensor product, see e.g.\ \cite[Proposition~2.5]{AS}
for the proof of this statement for $q=0$.

Let us first recall some standard notations from $q$-calculus.
For $n\in\N$,
\begin{align*}
[n]_q:=&1+q+q^2+\dots+q^{n-1}=\frac{1-q^n}{1-q},\\
 [n]_q!:=&[1]_q[2]_q\dotsm[n]_q=\frac{(1-q)(1-q^2)\dotsm(1-q^n)}{(1-q)^n},\quad [0]_q!:=1,\\
 {n\choose i}_q:=&\frac{[n]_q!}{[i]_q![n-i]_q!}\,,\quad i=0,1,\dots,n.
 \end{align*}

Below, for a real Hilbert space $\mathcal{H}$ and $c>0$,
we denote by $\mathcal{H}c$ the real Hilbert space which coincides with $\mathcal{H}$ as a set
and which has scalar product $(\cdot,\cdot)_{\mathcal{H}c}:=c(\cdot,\cdot)_\mathcal{H}$.

\begin{lemma}\label{ftyr7o}
Define the real Hilbert space
\[
\mathcal{G}_q(\mathcal{H}):=\bigoplus_{n=0}^\infty\mathcal{H}^{\otimes n}[n]_q!\,.
\]
Then $\mathcal{G}_q(\mathcal{H})$
is densely and continuously embedded into $\mathcal{F}_q(\mathcal{H})$.
\end{lemma}

\begin{proof}
We first note that $\mathcal{G}_{\mathrm{fin}}^{\mathrm a}(\mathcal{H})$
is a dense subset of both spaces $\mathcal{F}_q(\mathcal{H})$ and $\mathcal{G}_q(\mathcal{H})$.
By \cite[Remark on p.~525]{BS},
the norm of the bounded linear operator $P_q^{(n)}$ in $\mathcal{H}^{\otimes n}$
is equal to  $[n]_q!$\,.
Hence, for each $F=(f^{(n)})_{n=0}^\infty\in\mathcal{G}_{\mathrm{fin}}^{\mathrm a}(\mathcal{H})$,
it holds that
\[
\|F\|_{\mathcal{F}_q(\mathcal{H})}^2\le\sum_{n=0}^\infty
\|f^{(n)}\|_{\mathcal{H}^{\otimes n}}^2[n]_q!\,=\|F\|_{\mathcal{G}_q(\mathcal{H})}^2.
\]
This immediately implies the assertion.
\end{proof}

Analogously to $\mathcal{G}_q(\mathcal{H})$, we define,
for a real separable Hilbert space $\mathcal{H}$, $r>0$, and $\alpha\ge1$,  a real Hilbert space
\[
\mathcal{G}_q(\mathcal{H},r,\alpha)
:=\bigoplus_{n=0}^\infty\mathcal{G}^{(n)}_q(\mathcal{H},r,\alpha),
\quad \text{where }
\mathcal{G}^{(n)}_q(\mathcal{H},r,\alpha):=\mathcal{H}^{\otimes n}r^n([n]_q!)^\alpha.
\]

\begin{lemma}\label{yderi576}
Let $\mathcal{H}_+$ be a Hilbert space
which is densely embedded into $\mathcal{H}$
and satisfies $\|\cdot\|_\mathcal{H}\le\|\cdot\|_{\mathcal{H}_+}$.
(In particular, we may have $\mathcal{H}_+=\mathcal{H}$.)
Let $\alpha\ge1$ and let
\begin{equation}\label{yur786}
r\ge\max\{1,(1+q)^{1-\alpha}\}.
\end{equation}
Then, the Hilbert space $\mathcal{G}_q(\mathcal{H}_+,r,\alpha)$
is densely and continuously embedded into $\mathcal{F}_q(\mathcal{H})$.
\end{lemma}

\begin{proof}
In view of Lemma~\ref{ftyr7o},
it suffices to prove that, for all $n\in\mathbb N$,
\begin{equation}\label{ghdyy7}
r^n([n]_q!)^{\alpha-1}\ge1.
\end{equation}
If $q\in[0,1)$, then for each $n\in\mathbb N$, $[n]_q\ge1$,
and if $q\in(-1,0)$, then $[n]_q\ge1+q$.
From here inequality \eqref{ghdyy7} easily follows.
\end{proof}

Below, we will use the notions of a projective limit
and an inductive limit of Hilbert spaces.
For the definition, see e.g.\ \cite{BK}.

\begin{theorem}\label{cdre6i5}
Let $\mathcal{H}_+$ be a Hilbert space as in Lemma~\ref{yderi576}.
Let $\alpha\ge1$. We define
\begin{equation}\label{frtde6e}
\mathcal{G}_q(\mathcal{H}_+,\alpha):=\projlim_{r\ge1}\mathcal{G}_q(\mathcal{H}_+,r,\alpha).
\end{equation}
Then $\mathcal{G}_q(\mathcal{H}_+,\alpha)$
is densely and continuously embedded into $\mathcal{F}_q(\mathcal{H})$
and is an algebra under the addition and the tensor multiplication.
Furthermore, for any $s>r\ge1$, there exists $C_1>0$ such that
\begin{equation}\label{tye756ier7ir}
\|F\otimes G\|_{\mathcal{G}_q(\mathcal{H}_+,r,\alpha)}
\le C_1\|F\|_{\mathcal{G}_q(\mathcal{H}_+,s,\alpha)}\|G\|_{\mathcal{G}_q(\mathcal{H}_+,s,\alpha)}
\end{equation}
for any $F,G\in \mathcal{G}_q(\mathcal{H}_+,\alpha)$.
\end{theorem}

\begin{remark}
Evidently, when $\mathcal{H}_+=\mathcal{H}$,
the space $\mathcal{G}_q(\mathcal{H},\alpha)$
is an extension of $\mathcal{G}_{\mathrm{fin}}^{\mathrm a}(\mathcal{H})$.
\end{remark}

\begin{proof}[Proof of Theorem~\ref{cdre6i5}]
For any $F=(f^{(n)})_{n=0}^\infty,\,G=(g^{(n)})_{n=0}^\infty\in \mathcal{G}_{\mathrm{fin}}^{\mathrm a}(\mathcal{H}_+)$,
\begin{align}
\|F\otimes G\|^2_{\mathcal{G}_q(\mathcal{H}_+,r,\alpha)}
&=\sum_{n=0}^\infty\bigg\|\sum_{i=0}^n f^{(i)}\otimes g^{(n-i)}\bigg\|_{\mathcal{H}_+^{\otimes n}}^2\,r^n\,([n]_q!)^\alpha\notag\\
&\quad\le \sum_{n=0}^\infty\left(\sum_{i=0}^n\|f^{(i)}\|_{\mathcal{H}_+^{\otimes i}}\|g^{(n-i)}\|_{\mathcal{H}_+^{\otimes (n-i)}}\right)^2r^n\,([n]_q!)^\alpha\notag\\
 &\quad\le \sum_{n=0}^\infty r^n\,([n]_q!)^\alpha\,(n+1)\sum_{i=0}^n \|f^{(i)}\|^2_{\mathcal{H}_+^{\otimes i}}\|g^{(n-i)}\|_{\mathcal{H}_+^{\otimes (n-i)}}^2\,.\label{tye6ue}
\end{align}
Note that,  for each $n$,
\begin{equation}\label{tyd64e}
\|f^{(n)}\|_{\mathcal{H}_+^{\otimes n}}^2 \le \frac{\|F\|^2_{\mathcal{G}_q(\mathcal{H}_+,s,\alpha)}}{s^n\,([n]_q!)^\alpha}\,.
\end{equation}
Let $r<r_1<s$ and choose $C_2>0$ so that $r^n (n+1)\le C_2r_1^n$ for all $n$.
Then, by \eqref{tye6ue} and \eqref{tyd64e}, we get
\begin{align}
\|F\otimes G\|^2_{\mathcal{G}_q(\mathcal{H}_+,r,\alpha)}
&\le  C_2 \|F\|^2_{\mathcal{G}_q(\mathcal{H}_+,s,\alpha)}\|G\|^2_{\mathcal{G}_q(\mathcal{H}_+,s,\alpha)}\notag\\
&\qquad\times \left(\sum_{n=0}^\infty \bigg(\frac{r_1}s\bigg)^n \sum_{i=0}^n {n\choose i}_q ^\alpha\right)\,.
\label{vuytrt687ot}
\end{align}
Denote
\[
z:=\bigg(\frac{r_1}s\bigg)^{1/\alpha}.
\]
Then we obtain that
\[
\sum_{n=0}^\infty \bigg(\frac{r_1}s\bigg)^n \sum_{i=0}^n {n\choose i}_q ^\alpha
=\sum_{n=0}^\infty\sum_{i=0}^n\left[z^n{n\choose i}_q\right]^\alpha
   \le\left[\sum_{n=0}^\infty\sum_{i=0}^n z^n{n\choose i}_q\right]^\alpha.
\]
On the other hand, by using \cite[p.~17 and p.~36]{A}, we get
\begin{align*}
\sum_{n=0}^\infty\sum_{i=0}^n z^n{n\choose i}_q
&=\sum_{i=0}^\infty\sum_{n=i}^\infty z^n{n\choose i}_q
 =\sum_{i=0}^\infty\sum_{n=0}^\infty z^{n+i}{n+i\choose i}_q\notag\\
&=\sum_{n=0}^\infty z^n \sum_{i=0}^\infty z^i{n+i\choose i}_q\notag\\
&=\sum_{n=0}^\infty\frac{z^n}{(1-z)(1-zq)\dotsm(1-zq^{n})} \notag\\
&\le\sum_{n=0}^\infty\frac{z^n}{(1-z)(1-z|q|)\dotsm(1-z|q|^{n})} \notag\\
&\le\sum_{n=0}^\infty z^n\prod_{i=0}^\infty\frac1{1-z|q|^i}.
\end{align*}
Therefore, we see that
\begin{equation} \label{ftdstrse64u}
\sum_{n=0}^\infty \bigg(\frac{r_1}s\bigg)^n \sum_{i=0}^n {n\choose i}_q ^\alpha
\le \left[\sum_{n=0}^\infty z^n\prod_{i=0}^\infty\frac1{1-z|q|^i}\right]^\alpha.
\end{equation}
Thus, by \eqref{vuytrt687ot} and \eqref{ftdstrse64u},
we only need to prove that $\prod_{i=0}^\infty\frac1{1-z|q|^i}<\infty$.
But this holds since
\[
\sum_{i=0}^\infty\bigg(\frac1{1-z|q|^i}-1\bigg)
=\sum_{i=0}^\infty\frac{z|q|^i}{1-z|q|^i}\le\frac{z}{1-z}\sum_{i=0}^\infty |q|^i<\infty.
\]
\end{proof}

Recall that a real nuclear space $\Phi$ is defined as a projective limit
\begin{equation}\label{crte6u}
\Phi=\projlim_{\tau\in T} \mathcal{H}_\tau\,,
\end{equation}
where $T$ is an index set, $(H_\tau)_{\tau\in T}$
is a family of separable Hilbert spaces which are directed by embedding:
for any $\tau_1,\tau_2\in T$,
there exists $\tau_3\in T$ such that
the space $H_{\tau_3}$ is densely and continuously embedded into both  $H_{\tau_1}$ and $H_{\tau_2}$,
and furthermore, for each $\tau_1\in T$ there exists $\tau_2\in T$ such that
the embedding operator of $H_{\tau_2}$ into $H_{\tau_1}$ is of Hilbert--Schmidt class.
We define, for $\alpha\ge1$,
\begin{equation}\label{gyd66de}
\mathcal{G}_q(\Phi,\alpha):=\projlim_{(\tau,r)\in T\times[1,\infty)}\mathcal{G}_q(\mathcal{H}_\tau,r,\alpha).
\end{equation}
Analogously to \cite{KLS,KSWY}, one can show that $\mathcal{G}_q(\Phi,\alpha)$ is also a nuclear space.

\begin{corollary}\label{fdy7r6}
Let $\Phi$ be a nuclear space as in \eqref{crte6u}.
Assume that $\Phi$ is densely and continuously embedded into $\mathcal{H}$.
Then, for each $\alpha\ge1$, $\mathcal{G}_q(\Phi,\alpha)$
is densely and continuously embedded into $\mathcal{F}_q(\mathcal{H})$.
Furthermore, $\mathcal{G}_q(\Phi,\alpha)$ is an algebra with respect to the addition and the tensor multiplication,
with tensor multiplication being a continuous mapping from $\mathcal{G}_q(\Phi,\alpha)^2$ into $\mathcal{G}_q(\Phi,\alpha)$.
\end{corollary}

\begin{proof}
Analogously to $\mathcal{G}_{\mathrm{fin}}^{\mathrm a}(\mathcal{H})$,
we define the linear space $\mathcal{G}_{\mathrm{fin}}^{\mathrm a}(\Phi)$
which consists of all  finite vectors $(f^{(0)},f^{(1)},\dots,f^{(n)},0,0,\dots)$,
where $f^{(0)}\in\R$, $f^{(i)}\in\Phi^{\otimes_{\mathrm a}i}$ for $i=1,\dots,n$, $n\in\N$.
Since $\Phi$ is dense in $\mathcal{H}$,
$\Phi^{\otimes_{\mathrm a}n}$ is dense in $\mathcal{H}^{\otimes n}$ for each $n\in\N$.
Hence, $\mathcal{G}_{\mathrm{fin}}^{\mathrm a}(\Phi)$
is dense in $\mathcal{F}_q(\mathcal{H})$.
But $\mathcal{G}_{\mathrm{fin}}^{\mathrm a}(\Phi)\subset\mathcal{G}_q(\Phi,\alpha)$.
Hence $\mathcal{G}_q(\Phi,\alpha)$ is dense in $\mathcal{F}_q(\mathcal{H})$.
The other statements of Corollary~\ref{fdy7r6} follow from Theorem~\ref{cdre6i5}.
\end{proof}

\begin{remark}\label{futfr687r}
Let $q\in(-1,0)$.
The obvious inequality $[n]_q\le[n]_{|q|}$ implies that,
instead of the spaces $\mathcal{G}_q(\mathcal{H}_+,\alpha)$
and $\mathcal{G}_q(\Phi,\alpha)$ in Theorem~\ref{cdre6i5} and Corollary~\ref{fdy7r6}, respectively,
we can use the smaller spaces $\mathcal{G}_{|q|}(\mathcal{H}_+,\alpha)$
and $\mathcal{G}_{|q|}(\Phi,\alpha)$.
We will use this observation below.
\end{remark}

\section{Wick calculus for generalized  functionals of $q$-white noise}\label{jkgt7u6r7}

In view of the unitary operator $I:L^2(\mu)\to\mathcal{F}_q(\mathcal{H})$,
Theorem~\ref{cdre6i5}, Corollary~\ref{fdy7r6}, and Remark~\ref{futfr687r},
$\mathcal{G}_{|q|}(\mathcal{H}_+,\alpha)$ and $\mathcal{G}_{|q|}(\Phi,\alpha)$
may be thought of as spaces of test functionals of noncommutative $q$-white noise $\omega$.
Our next aim is to extend the Wick product (equivalently the tensor product) to the dual spaces,
which may be thought of as spaces of generalized functionals of $\omega$.

For a real separable Hilbert space $\mathcal{H}$ and $n\in\N$,
we define $\mathbb{F}_q^{(n)}(\mathcal{H})$ as the Hilbert space
which coincides with $\mathcal{H}^{\otimes n}$ as a set and which has scalar product
\[
( f^{(n)},g^{(n)})_{\mathbb{F}^{(n)}_q(\mathcal{H})}
:=( P_q^{(n)} f^{(n)},P_q^{(n)}g^{(n)})_{\mathcal{H}^{\otimes n}}\,.
\]
In particular,
\[
\|f^{(n)}\|_{\mathbb{F}_q^{(n)}(\mathcal{H})}=\|P_q^{(n)}f^{(n)}\|_{\mathcal{H}^{\otimes n}}.
\]
We also set $\mathbb{F}_q^{(0)}(\mathcal{H}):=\R$.
For $r\ge1$ and $\alpha\in\R$, we then define a Hilbert space
\[
\mathbb{F}_q(\mathcal{H},r,\alpha):=\bigoplus_{n=0}^\infty
\mathbb{F}^{(n)}_q(\mathcal{H},r,\alpha),
~~ \text{where }
\mathbb{F}^{(n)}_q(\mathcal{H},r,\alpha):=
\mathbb{F}_q^{(n)}(\mathcal{H}) r^n([n]_{|q|}!)^\alpha.
\]
Note that we use $[n]_{|q|}!$ rather than $[n]_q!$
in the definition of $\mathbb{F}_q(\mathcal{H},r,\alpha)$.

\begin{proposition}\label{tyri745r}
Let $\mathcal{H}_+$ be a Hilbert space as in Lemma~\ref{yderi576}.
Consider the standard triple of Hilbert spaces
\[
\mathcal{H}_+\subset\mathcal{H}\subset\mathcal{H}_-\,,
\]
where $\mathcal{H}_-$ is the dual space of $\mathcal{H}_+$ with respect to the center space $\mathcal{H}$,
i.e., the dual pairing between elements of $\mathcal{H}_-$ and $\mathcal{H}_+$
is given by an extension of the scalar product in $\mathcal{H}$.
Let $\alpha\ge1$ and let $r\ge 1$.
Then we get the standard triple of Hilbert spaces
\[
\mathcal{G}_{|q|}(\mathcal{H}_+,r,\alpha)
\subset \mathcal{F}_q(\mathcal{H})
\subset \mathbb{F}_q(\mathcal{H}_-\,,r^{-1},-\alpha).
\]
\end{proposition}

\begin{proof}
By Lemma~\ref{yderi576} and Remark~\ref{futfr687r},
the Hilbert space $\mathcal{G}_{|q|}^{(n)}(\mathcal{H}_+,r,\alpha)$
is densely and continuously embedded into $\mathcal{F}^{(n)}_q(\mathcal{H})$.
Hence, using the construction of a rigged Hilbert space,
see e.g.\ \cite[Chapter~1, Section~1]{BK},
we easily obtain the  standard triple
\[
\mathcal{G}_{|q|}^{(n)}(\mathcal{H}_+,r,\alpha)
\subset\mathcal{F}_q^{(n)}(\mathcal{H})
\subset\mathbb{F}^{(n)}_q(\mathcal{H}_-,r^{-1},-\alpha).
\]
From here the statement follows.
\end{proof}

Proposition~\ref{tyri745r} implies that the dual of the space
$\mathcal{G}_{|q|}(\mathcal{H}_+,\alpha)$ with respect to the center space $\mathcal{F}_q(\mathcal{H})$
is the space
\[
\operatornamewithlimits{ind\,lim}_{r\ge1}\mathbb{F}_q(\mathcal{H}_-,r^{-1},-\alpha)
=:\mathbb{F}_q(\mathcal{H}_-,-\alpha).
\]

\begin{theorem}\label{rte6u4}
The space $\mathbb{F}_q(\mathcal{H}_-\,,-2)$
is an algebra under the addition and the tensor multiplication.
Furthermore, for any $r>s\ge1$, we have
\begin{equation}\label{cfxtdr6ue7i}
\|F\otimes G\|_{\mathbb{F}_q(\mathcal{H}_-,r^{-1},-2)}\le \bigg(\frac{r}{r-s}\bigg)^{1/2} \|F\|_{\mathbb{F}_q(\mathcal{H}_-,s^{-1},-2)}\|G\|_{\mathbb{F}_q(\mathcal{H}_-,r^{-1},-2)}\,,
\end{equation}
where $F\in \mathbb{F}_q(\mathcal{H}_-,s^{-1},-2)$,
$G\in \mathbb{F}_q(\mathcal{H}_-,r^{-1},-2)$.
\end{theorem}

\begin{remark}
There is an important difference between formulas \eqref{tye756ier7ir} and \eqref{cfxtdr6ue7i}:
in the latter formula, one uses the same norm for  $F\otimes G$ and $G$,
namely the norm of the space $\mathbb{F}_q(\mathcal{H}_-,r^{-1},-2)$.
This observation will be crucial for the proof of Theorem~\ref{vuf7t} below.
\end{remark}

\begin{remark}
By taking the limit as $q\to 1$,
we easily conclude that inequality \eqref{cfxtdr6ue7i}
also holds in the classical (commutative) case $q=1$
for the space $\mathbb{F}_1(\mathcal{H}_-,-2)$
in which the corresponding $\mathbb{F}_1^{(n)}(\mathcal{H}_-)$ spaces
consist of symmetrized elements $(n!)^{-1}P_1^{(n)}f^{(n)}$ with $f^{(n)}\in\mathcal{H}_-^{\otimes n}$.
(The operator $(n!)^{-1}P_1^{(n)}$ is the symmetrization projection.)
In particular, choosing $\mathcal{H}_-=\mathcal{H}$,
we get inequality \eqref{cfxtdr6ue7i} on $\mathbb{F}_1(\mathcal{H},-2)$.
The latter space is the Kondratiev space of regular generalized functions,
constructed by Grothaus, Kondratiev and Streit \cite{GKL}, see also \cite{GKU}.
Formula \eqref{cfxtdr6ue7i} is then a version of the V\r{a}ge inequality,
originally derived in \cite{V} (see also \cite{HOUZ}),
for the Kondratiev space of (non-regular) generalized functions \cite{K,KLS,KSWY}.
Note that, for the space $\mathbb{F}_1(\mathcal{H},-2)$,
Theorem~\ref{vuf7t} below was proven in \cite{GKU} by different methods,
without the use of a V\r{a}ge-type inequality.
So, even in the commutative case $q=1$, inequality \eqref{cfxtdr6ue7i} is a new result.

Alpay and Salomon \cite{AS2012} introduced a concept of a V\r{a}ge space
on which a V\r{a}ge-type inequality holds.
In the noncommutative, free setting ($q=0$),
a V\r{a}ge-type inequality was derived in \cite{AS2013,AS},  see also \cite{AJS}.
Note that, even for $q=0$, the form of inequality \eqref{cfxtdr6ue7i}
significantly differs from the result in \cite{AS2013,AS}.
\end{remark}

\begin{remark}
Let $G=\mathbb{Z}$ be the additive group of integers and consider its  semigroup 
$S=\mathbb{N}_0=\{0,1,2,\dots\}\subset G$. Let $\mu$ be the counting measure on  $G$. For any $r\ge 1$, define the measure $\mu_r$ on $S$ by 
\[
\frac{d\mu_r}{d\mu}(n):=\left([n]_{|q|}!\right)^{-2}r^{-n}.
\]
Then for any $r>s\ge1$, it holds that 
\[
\int_{S} \frac{d\mu_r}{d\mu_s}(n)\,d\mu(n)=\frac{r}{r-s}
\]
and for any $m,n\in S$,
\begin{align*}
\frac{d\mu_r}{d\mu}(n+m)
&=\left([n+m]_{|q|}!\right)^{-2}r^{-(n+m)}
\le \left([n]_{|q|}!\,[m]_{|q|}!\right)^{-2}r^{-(n+m)}\\
&=\frac{d\mu_r}{d\mu}(n)\frac{d\mu_r}{d\mu}(m).
\end{align*}
Therefore, in the case where $\mathcal H_-$ is a one-dimensional real Hilbert space (equivalently $\mathcal H_-=\mathbb R$),  
the result of Theorem \ref{rte6u4} follows from 
Theorem~3.4 in Alpay and Salomon \cite{AS2013}.

Note that Theorem 3.4 in  \cite{AS2013} implies both the  
V\r{a}ge inequality  in the classical setting and the V\r{a}ge-type inequality in the free setting, see Section~6 in \cite{AS2013}. This is shown by a proper choice of a discrete group $G$, a semigroup $S\subset G$, and measures $\mu_r$  on $S$ ($r\in\mathbb N$). In particular, for each $s\in\mathbb N$, there should exist $r\in\mathbb N$, $r>s$, such that
\begin{equation}\label{tyer7o}
\int_S \frac{d\mu_r}{d\mu_s}(\alpha)\,d\mu(\alpha)<\infty,\end{equation}
where $\mu$ is the counting measure. But condition \eqref{tyer7o} implies that the operator of embedding of $L^2(S,\mu_s)$ into $L^2(S,\mu_r)$ is of Hilbert--Schmidt class.  However, in the case of an infinite-dimensional space $\mathcal H_-$ and $r>s\ge 1$, the operator of embedding of  $\mathbb{F}_q(\mathcal{H}_-,s^{-1},-2)$ into $\mathbb{F}_q(\mathcal{H}_-,r^{-1},-2)$ is not of Hilbert--Schmidt class. Hence, in the general case, the result of Theorem \ref{rte6u4} does not  follow from 
 \cite[Theorem~3.4]{AS2013}.  
\end{remark}

\begin{proof}[Proof of Theorem \ref{rte6u4}]
We will first prove the following lemma, which can be of independent interest.

\begin{lemma}\label{jkigiy8tjbg}
Let $f^{(m)}\in\mathcal{H}^{\otimes m}$, $g^{(n)}\in\mathcal{H}^{\otimes n}$, $m,n\in\mathbb N$.
Then
\begin{equation}\label{iugti9t}
\|f^{(m)}\otimes g^{(n)}\|_{\mathbb{F}_q^{(m+n)}(\mathcal{H})}\le {m+n\choose m}_{|q|}\,
\|f^{(m)}\|_{\mathbb{F}_q^{(m)}(\mathcal{H})}
\|g^{(n)}\|_{\mathbb{F}_q^{(n)}(\mathcal{H})}.
\end{equation}
\end{lemma}

\begin{proof}
We denote by $\mathfrak S(m+n,m)$ the collection of all subsets
\[
A=\{i_1,i_2,\dots,i_m\}\subset \{1,2,\dots,m+n\}.
\]
We assume that $i_1<i_2<\dots<i_m$. Let also
\[
\{1,2,\dots,m+n\}\setminus A=\{j_1,j_2,\dots,j_n\}
\]
with $j_1<j_2<\dots<j_n$.
We denote by $\operatorname{inv}(A)$ the number of all pairs $(j_k,i_l)$ such that $j_k<i_l$.

For each permutation $\pi\in S_{m+n}$,
there exist unique permutations $\theta\in S_m$, $\nu\in S_n$,
and a set $A\in \mathfrak S(m+n,m)$ as above such that
\begin{gather}
\pi(i_1)=\theta(1),\ \pi(i_2)=\theta(2),\dots,  \pi(i_m)=\theta(m),\notag\\
\pi(j_1)=m+\nu(1),\ \pi(j_2)=m+\nu(2),\dots,\pi(j_n)=m+\nu(n).
\label{bgu8yt8}
\end{gather}
In fact, we first construct a set $A\in \mathfrak S(m+n,m)$ by
\[
A:=\pi^{-1}(\{1,2,\cdots,m\}):=\{i_1,i_2,\cdots,i_m\},
\]
and then the permutations $\theta\in S_m$, $\nu\in S_n$ are constructed as given in \eqref{bgu8yt8}.
Inversely, any $\theta\in S_m$, $\nu\in S_n$,
and $A\in\mathfrak S(m+n,m)$ determine by formula \eqref{bgu8yt8} a permutation $\pi\in S_{m+n}$.
Note that
\begin{equation}\label{hjgftr7e6e}
\operatorname{inv}(\pi)=\operatorname{inv}(\theta)+\operatorname{inv}(\nu)+\operatorname{inv}(A).
\end{equation}

For each $A\in\mathfrak S(m+n,m)$ as above,
we define a unitary operator $U(A):\mathcal{H}^{\otimes(m+n)}\to\mathcal{H}^{\otimes(m+n)}$ by setting,
for any $f_1,f_2,\dots,f_{m+n}\in\mathcal{H}$,
\[
U(A)f_1\otimes f_2\otimes\dots\otimes f_{m+n}:=g_1\otimes g_2\otimes\dots\otimes g_{m+n},
\]
where
\begin{gather*}
g_{i_1}=f_1,\ g_{i_2}=f_2,\dots,g_{i_m}=f_m,\\
g_{j_1}=f_{m+1},\ g_{j_2}=f_{m+2},\dots, g_{j_n}=f_{m+n}
\end{gather*}
(i.e., $U(A)$ swaps the vectors in the tensor product according to $A$
and preserving the order of $f_1,\dots,f_m$ and of $f_{m+1},\dots,f_{m+n}$).
Then, for $\pi\in S_{m+n}$ given by \eqref{bgu8yt8}, we get
\begin{align}\label{gdf7yur87}
&f_{\pi(1)}\otimes f_{\pi(2)}\otimes\dots\otimes f_{\pi(m+n)}\\
&=U(A) f_{\theta(1)}\otimes f_{\theta(2)}\otimes\dots\otimes
f_{\theta(m)}\otimes f_{m+\nu(1)}\otimes f_{m+\nu(2)}\otimes\dots\otimes f_{m+\nu(n)}.
  \nonumber
\end{align}
Thus, by \eqref{hvftdf7yihyu}, \eqref{hjgftr7e6e}, and \eqref{gdf7yur87},
we obtain that
\begin{align*}
&P_q^{(m+n)}f_1\otimes f_2\otimes\dots\otimes f_{m+n}\\
&~~
=\sum_{A\in\mathfrak S(m+n,m)}
q^{\operatorname{inv}(A)}\,
U(A)\left[\left(\sum_{\theta\in S_m}q^{\operatorname{inv}(\theta))} f_{\theta(1)}\otimes f_{\theta(2)}\otimes\dots\otimes
f_{\theta(m)}\right)\right.\\
&\hspace{25mm}\otimes \left.\left(\sum_{\nu\in S_n} q^{\operatorname{inv}(\nu))} f_{m+\nu(1)}\otimes f_{m+\nu(2)}\otimes\dots\otimes f_{m+\nu(n)}\right)\right]\\
&~~=\sum_{A\in\mathfrak S(m+n,m)}q^{\operatorname{inv}(A)}\,U(A)
\bigl[(P_q^{(m)}f_1\otimes f_2\otimes\dots\otimes f_m)\bigr.\\
&\hspace{25mm}\otimes \bigl.(P_q^{(n)}f_{m+1}\otimes f_{m+2}\otimes\dots\otimes f_{m+n})\bigr].
\end{align*}
Hence
\begin{align}
&\|P_q^{(m+n)}f_1\otimes f_2\otimes\dots\otimes f_{m+n}\|_{\mathcal{H}^{\otimes(m+n)}}\notag\\
&\qquad \le \left(\sum_{A\in\mathfrak S(m+n,m)}|q|^{\operatorname{inv}(A)}\right)
  \|P_q^{(m)}f_1\otimes f_2\otimes\dots\otimes f_m\|_{\mathcal{H}^{\otimes m}}\,\notag\\
&\qquad \qquad\times\|P_q^{(n)}f_{m+1}\otimes f_{m+2}\otimes\dots\otimes f_{m+n}\|_{\mathcal{H}^{\otimes n}}.
\label{gigy8iy}
\end{align}
By MacMahon theorem, see e.g.\ \cite[Section~3.4]{A},
\begin{equation} \sum_{A\in\mathfrak S(m+n,m)}
|q|^{\operatorname{inv}(A)}={m+n\choose m}_{|q|}.\label{tdy6e64}
\end{equation}
Thus, by \eqref{gigy8iy} and \eqref{tdy6e64},
inequality \eqref{iugti9t} holds for $f^{(m)}=f_1\otimes\dots\otimes f_m$, $g^{(n)}=f_{m+1}\otimes\dots\otimes f_{m+n}$.
The case of general $f^{(m)}$ and $g^{(n)}$ can be done by analogy.
 \end{proof}

Let $r>s\ge1$ and let $F=(f^{(n)})_{n=0}^\infty$,
$G=(g^{(n)})_{n=0}^\infty\in \mathcal{G}_{\mathrm{fin}}^{\mathrm a}(\mathcal{H}_-)$.
Then by using Lemma~\ref{jkigiy8tjbg}, we obtain that
\begin{align*}
&\|F\otimes G\|_{\mathbb{F}(\mathcal{H}_-,r^{-1},-2)}^2=\sum_{n=0}^\infty\bigg\|
\sum_{i=0}^n f^{(i)}\otimes g^{(n-i)}
\bigg\|^2_{\mathbb{F}_q^{(n)}(\mathcal{H}_-)}r^{-n}\big([n]_{|q|}!\big)^{-2}\\
&\quad\le\sum_{n=0}^\infty\left(\sum_{i=0}^n
\|f^{(i)}\otimes g^{(n-i)}\|_{\mathbb{F}_q^{(n)}(\mathcal{H}_-)}\right)^2 r^{-n}\big([n]_{|q|}!\big)^{-2}\\
&\quad\le\sum_{n=0}^\infty\left(\sum_{i=0}^n {n\choose i}_{|q|}
\|f^{(i)}\|_{\mathbb{F}_q^{(i)}(\mathcal{H}_-)}\|g^{(n-i)}\|_{\mathbb{F}_q^{(n-i)}(\mathcal{H}_-)}
\right)^2 r^{-n}\big([n]_{|q|}!\big)^{-2}.
\end{align*}
Put
\[
a_i=\|f^{(i)}\|_{\mathbb{F}_q^{(i)}(\mathcal{H}_-)},
\qquad
b_{n-i}=\|g^{(n-i)}\|_{\mathbb{F}_q^{(n-i)}(\mathcal{H}_-)}.
\]
Then by using the Cauchy inequality,
we obtain that
\begin{align*}
&\|F\otimes G\|_{\mathbb{F}(\mathcal{H}_-,r^{-1},-2)}^2\\
&\quad\le \sum_{n=0}^\infty\left(\sum_{i=0}^n {n\choose i}_{|q|} a_ib_{n-i}\right)^2 r^{-n}\big([n]_{|q|}!\big)^{-2}\\
&\quad =\sum_{n=0}^\infty \sum_{i,j=0}^n\left(
  [i]_{|q|}!\, [j]_{|q|}!\, [n-i]_{|q|}!\,[n-j]_{|q|}!\right)^{-1} a_ib_{n-i}a_j b_{n-j}r^{-n}\\
&\quad=\sum_{i,j=0}^\infty a_ia_j\left([i]_{|q|}!\, [j]_{|q|}!\right)^{-1}\\
  &\qquad\qquad\times \sum_{n\ge\max\{i,j\}}b_{n-i}b_{n-j}\left([n-i]_{|q|}!\,[n-j]_{|q|}!\right)^{-1}r^{-n}\\
&\quad\le \sum_{i,j=0}^\infty a_ia_j\left([i]_{|q|}!\, [j]_{|q|}!\right)^{-1}
    \left(\sum_{n\ge\max\{i,j\}}b_{n-i}^2\left([n-i]_{|q|}!\right)^{-2}r^{-n}\right)^{1/2}\\
  &\qquad\qquad\times\left(\sum_{n\ge\max\{i,j\}}b_{n-j}^2\left([n-j]_{|q|}!\right)^{-2}r^{-n}\right)^{1/2}\\
 &\quad\le \left(\sum_{i=0}^\infty a_i\left([i]_{|q|}!\right)^{-1}r^{-i/2}\right)^2
\left(\sum_{n=0}^\infty b_n^2\left([n]_{|q|}!\right)^{-2}r^{-n}\right)\\
&\quad\le\left(\sum_{i=0}^\infty\left(\frac{r}{s}\right)^{-i}\right)
  \|F\|^2_{\mathbb{F}_q(\mathcal{H},s^{-1},-2)}\|G\|^2_{\mathbb{F}_q(\mathcal{H},r^{-1},-2)}.
\end{align*}
From here the theorem follows.
\end{proof}

Analogously to \cite[Section 4]{AS},
we will now derive several consequences of inequality~\eqref{cfxtdr6ue7i}.

\begin{theorem}\label{vuf7t}
Let $\phi(z)=\sum_{n=0}^\infty a_nz^n$ be an analytic function of a complex variable $z$
defined in a neighborhood of zero.
Assume that the series $\sum_{n=0}^\infty a_nz^n$
converges absolutely in the open disk of radius $R>0$.
Then, for each $F=(f^{(n)})_{n=0}^\infty\in\mathbb{F}_q(\mathcal{H}_-\,,-2)$ such that $|f^{(0)}|<R$,
we have
\[
\phi^{\otimes}(F):=\sum_{n=0}^\infty a_n F^{\otimes n}\in
\mathbb{F}_q(\mathcal{H}_-\,,-2)_{\mathbb C}.
\]
Here $\mathbb{F}_q(\mathcal{H}_-\,,-2)_{\mathbb C}$
denotes the complexification of $\mathbb{F}_q(\mathcal{H}_-\,,-2)$.
More precisely, there exists $r\ge 1$ such that
the series $\sum_{n=0}^\infty a_n F^{\otimes n}$
converges in $\mathbb{F}_q(\mathcal{H}_-\,,r^{-1},-2)_{\mathbb C}$\,.
\end{theorem}

\begin{proof}
Choose $s\ge1$ so that $F\in \mathbb{F}_q(\mathcal{H}_-\,,s^{-1},-2)$.
Since  $|f^{(0)}|<R$, by choosing $s$ sufficiently large,
we may assume that $\|F\|_{\mathbb{F}_q(\mathcal{H}_-\,,s^{-1},-2)}<R$.
Using \eqref{cfxtdr6ue7i}, we can easily show by induction that, for each $r>s$ and $n\ge2$,
\[
\|F^{\otimes n}\|_{\mathbb{F}_q(\mathcal{H}_-\,,r^{-1},-2)}
\le\bigg(\frac{r}{r-s}\bigg)^{(n-1)/2}\|F\|^{n}_{_{\mathbb{F}_q(\mathcal{H}_-\,,s^{-1},-2)}}.
\]
Let $\varepsilon:=R-\|F\|_{\mathbb{F}_q(\mathcal{H}_-\,,s^{-1},-2)}>0$.
Then,
\[
\|F^{\otimes n}\|_{\mathbb{F}_q(\mathcal{H}_-\,,r^{-1},-2)}
\le \bigg(\frac{r}{r-s}\bigg)^{(n-1)/2} (R-\varepsilon)^n.
\]
Hence, the statement follows if we choose $r$ so large that
\[
\bigg(\frac{r}{r-s}\bigg)^{1/2}(R-\varepsilon)<R.
\]
\end{proof}

\begin{corollary}\label{tye5674}
Let $\phi(z)$ be an analytic function of a complex variable $z$ defined in a neighborhood of zero.
Then, for each $F\in\mathbb{F}_q(\mathcal{H}_-\,,-2)$,
we have
$\phi^{\otimes}(zF)\in\mathbb{F}_q(\mathcal{H}_-\,,-2)_{\mathbb C}$
for all complex $z$ from a neighborhood of zero.
\end{corollary}

Analogously to \cite[Proposition~4.10]{AS}, we also obtain

\begin{corollary}\label{hufrku7r87o}
Let $F=(f^{(n)})_{n=0}^\infty\in\mathbb{F}_q(\mathcal{H}_-\,,-2)$.
Then $F$ has an inverse element $F^{\otimes(-1)}\in\mathbb{F}_q(\mathcal{H}_-\,,-2)$
with respect to tensor multiplication if and only if $f^{(0)}\ne0$.
In the latter case,
\[
F^{\otimes(-1)}=(f^{(0)})^{-1}\sum_{n=0}^\infty\big(\Omega-(f^{(0)})^{-1}F\big)^{\otimes n}.
\]
\end{corollary}

Finally, let us briefly discuss the case of a nuclear space.
Let $\Phi$ be a nuclear space as in \eqref{crte6u}.
Assume that $\Phi$ is densely and continuously embedded into $\mathcal{H}$.
Without loss of generality,
we may assume that each Hilbert space $\mathcal{H}_\tau$ ($\tau \in T$)
is a dense subset of $\mathcal{H}$
and $\|\cdot\|_\mathcal{H}\le\|\cdot\|_{\mathcal{H}_\tau}$.
Denote by $\mathcal{H}_{-\tau}$ the dual space of $\mathcal{H}_\tau$
with respect to the center space $\mathcal{H}$.
We then get a Gel'fand triple
\[
\Phi\subset \mathcal{H}\subset\Phi',
\]
where the dual space $\Phi'$ has representation
\[
\Phi'=\operatornamewithlimits{ind\,lim}_{\tau\in T}\mathcal{H}_{-\tau},
\]
see e.g.\ \cite{BK}.

Let us consider the  nuclear space $\mathcal{G}_{|q|}(\Phi,2)$, see \eqref{gyd66de}.
We also define
\[
\mathbb{F}_q(\Phi',-2)
:=\operatornamewithlimits{lim\,ind}_{(\tau,r)\in T\times[1,\infty)}\mathbb{F}_q(\mathcal{H}_{-\tau},r^{-1},-2).
\]
By Proposition~\ref{tyri745r},
we then get the Gel'fand triple
\[
\mathcal{G}_{|q|}(\Phi,2)\subset\mathcal{F}_q(\mathcal{H})\subset\mathbb{F}_q(\Phi',-2),
\]
where the co-nuclear space $\mathbb{F}_q(\Phi',-2)$
is the dual of $\mathcal{G}_{|q|}(\Phi,2)$
with respect to the center space $\mathcal{F}_q(\mathcal{H})$.

\begin{corollary}\label{guf7t}
The  co-nuclear space $\mathbb{F}_q(\Phi',-2)$
is an algebra under addition and tensor multiplication.
\end{corollary}

\begin{proof}
Immediate from Theorem~\ref{rte6u4}.
\end{proof}

Clearly, the results of Theorem~\ref{vuf7t} and Corollaries~\ref{tye5674}, \ref{hufrku7r87o}
admit straightforward generalization to the case of $\mathbb{F}_q(\Phi',-2)$.

\begin{center}
{\bf Acknowledgements}\end{center}
 
We would like to thank the anonymous referee for their careful reading of our manuscript and bringing the paper \cite{AS2013} to our attention.

\end{document}